\documentclass[12pt,a4paper,reqno]{amsart}
\usepackage{amsmath,amsfonts,amssymb,amsthm,amscd}
\usepackage{pst-node,pstricks,pst-tree}

\usepackage{graphicx}
\usepackage{xcolor}

\input{comment.sty}
\includecomment{MM}
\includecomment{CL}

\renewcommand{\a}{\alpha}
\renewcommand{\b}{\beta}
\newcommand{\g}{\gamma}

\renewcommand{\d}{\delta}

\renewcommand{\l}{\lambda}

\newcommand{\m}{\mu}

\renewcommand{\o}{\omega}
\renewcommand{\O}{\Omega}
\renewcommand{\r}{\rho}
\newcommand{\s}{\sigma}

\renewcommand{\t}{\tau}

\newcommand{\e}{\varepsilon}
\newcommand{\f}{\varphi}
\newcommand{\F}{\Phi}

\newcommand{\Ac}{{\mathcal A}}

\newcommand{\Gc}{{\mathcal G}}
\newcommand{\Fc}{{\mathcal F}}
\renewcommand{\Mc}{{\mathcal M}}
\newcommand{\Hc}{{\mathcal H}}
\newcommand{\Qc}{{\mathcal Q}}

\newcommand{\C}{{\mathbb C}}
\newcommand{\E}{{\mathbb E}}

\newcommand{\Xc}{{\mathcal X}}

\newcommand{\Sc}{{\mathcal S}}

\newcommand{\R}{{\mathbb R}}

\newcommand{\Z}{{\mathbb Z}}
\renewcommand{\Pr}{{\mathbb P}}

\renewcommand{\Re}{{\rm Re}}

\newcommand{\TV}{{\rm TV}}
\newcommand{\Var}{{\rm Var}}

\newcommand{\wh}[1]{\widehat{#1}}
\newcommand{\wt}[1]{\widetilde{#1}}
\newcommand{\1}{{\bf 1}}

\newcommand{\id}{{\rm id}}

\newcommand{\HS}{{\rm HS}}
\newcommand{\End}{{\rm End}}
\renewcommand{\Tr}{{\rm Tr}\,}

\newtheorem{theorem}{Theorem}[section]

\newtheorem{lemma}[theorem]{Lemma}

\theoremstyle{definition}

\theoremstyle{remark}
\newtheorem{remark}[theorem]{Remark}

\newcommand{\floor}[1]{\left\lfloor #1 \right\rfloor}
\newcommand{\ceil}[1]{\left\lceil #1 \right\rceil}
\newcommand{\normTV}[1]{\left\| #1 \right\|_{\TV}}
\newcommand{\Sast}[1]{\Sc_\ast\!\left( #1 \right)}
\newcommand{\Snon}[1]{\Sc_{non}\!\left( #1 \right)}

\begin{document}

\title[Cutoff for product replacement on finite groups]{Cutoff for product replacement on finite groups}
\author{Yuval Peres, Ryokichi Tanaka, and Alex Zhai}
\address{Yuval Peres, Microsoft Research}
\email{peres@microsoft.com}
\address{Ryokichi Tanaka, Tohoku University}
\email{rtanaka@m.tohoku.ac.jp}
\address{Alex Zhai, Stanford University}
\email{azhai@stanford.edu}

\date{\today}

\begin{abstract}
  We analyze a Markov chain, known as the \emph{product replacement
    chain}, on the set of generating $n$-tuples of a fixed finite
  group $G$. We show that as $n \rightarrow \infty$, the
  total-variation mixing time of the chain has a cutoff at time
  $\frac{3}{2} n \log n$ with window of order $n$. This generalizes a
  result of Ben-Hamou and Peres (who established the result for $G =
  \Z/2$) and confirms a conjecture of Diaconis and Saloff-Coste that
  for an arbitrary but fixed finite group, the mixing time of the
  product replacement chain is $O(n \log n)$.
\end{abstract}

\maketitle

\section{Introduction}

Let $G$ be a finite group, and let $[n] := \{1, 2, \dots, n\}$. We
consider the set $G^n$ of all functions $\sigma: [n] \to G$ (or
``configurations''). We may define a Markov chain $(\s_t)_{t\ge0}$ on
$G^{n}$ as follows: if we have a current state $\s$, then uniformly at
random, choose an ordered pair $(i, j)$ of distinct integers in $[n]$,
and change the value of $\s(i)$ to $\s(i) \s(j)^{\pm 1}$, where the
signs are chosen with equal probability.

We will restrict the chain $(\s_t)_{t\ge 0}$ to the space of {\it
  generating $n$-tuples}, i.e.\ the set of $\s$ whose values generate
$G$ as a group:
\[ \Sc := \left\{ \s \in G^n \ : \ \langle \s(1), \dots, \s(n) \rangle=G \right\}. \]
It is not hard to see that for fixed $G$ and large enough $n$, the
chain on $\Sc$ is irreducible (see \cite[Lemma 3.2]{DSC96}). We will
always assume $n$ is large enough so that this irreducibility
holds. Note that the chain is also symmetric, and it is aperiodic
because it has holding on some states. Thus, the chain has a uniform
stationary distribution $\pi$ with $\pi(\s)=1/|\Sc|$.

This Markov chain was first considered in the context of computational
group theory---it models the \emph{product replacement algorithm} for
generating random elements of a finite group introduced in
\cite{Celler}. By running the chain for a long enough time $t$ and
choosing a uniformly random index $k \in [n]$, the element $\s_t(k)$
is a (nearly) uniformly random element of $G$. The product replacement
algorithm has been found to perform well in practice \cite{Celler,
  Holt-Rees}, but the question arises: how large does $t$ need to be
in order to ensure near uniformity?

One way of answering the question is to estimate the mixing time of
the Markov chain. It was shown by Diaconis and Saloff-Coste that for
any fixed finite group $G$, there exists a constant $C_G$ such that
the $\ell^2$-mixing time is at most $C_G n^2 \log n$ \cite{DSC96,
  DSC98} (see also Chung and Graham \cite{Chung-Graham-Group} for a
simpler proof of this fact with a different value for $C_G$).

In another line of work, Lubotzky and Pak \cite{Lubotzky-Pak} analyzed
the mixing of the product replacement chain in terms of Kazhdan
constants (see also subsequent quantitative estimates for Kazhdan
constants by Kassabov \cite{Kassabov}). We also mention a result of
Pak \cite{Pak} which shows mixing in $\text{polylog}(|G|)$ steps when
$n = \Theta(\log |G| \log \log |G|)$. The reader may consult the
survey \cite{Pak-Survey} for further background on the product
replacement algorithm.

Diaconis and Saloff-Coste conjectured that the mixing time bound can
be improved to $C_G n \log n$ \cite[Remark 2, Section 7,
  p.\ 290]{DSC98}, based on the observation that at least $n \log n$
steps are needed by the classical coupon-collector's problem. This was
confirmed in the case $G = \Z/2$ by Chung and Graham
\cite{Chung-Graham-Cube} and recently refined by Ben-Hamou and Peres,
who show that when $G=\Z/2$, the chain in fact exhibits a cutoff at
time $\frac{3}{2}n \log n$ in total-variation with window of order $n$
\cite{BHP}.

In this paper, we extend the result of Ben-Hamou and Peres to all
finite groups. Note that this also verifies the conjecture of Diaconis
and Saloff-Coste for a fixed finite group. To state the result, let us
denote the total variation distance between $\Pr_\s(\s_t \in \cdot
\ )$ and $\pi$ by
\[ d_\s(t) := \max_{A \subseteq \Sc}|\Pr_\s(\s_t \in A)- \pi (A)|. \]

\begin{theorem}\label{Thm:main}
Let $G$ be a finite group. Then, the Markov chain $(\s_t)_{t \ge 0}$ on the set of
generating $n$-tuples of $G$ has a total-variation cutoff at time
$\frac{3}{2}n\log n$ with window of order $n$. More precisely, we have
\begin{equation}\label{Eq:UB}
  \lim_{\b\to\infty} \limsup_{n\to\infty} \max_{\s \in \Sc}
  d_\s\left(\frac{3}{2}n\log n + \b n\right) = 0
\end{equation}
and
\begin{equation}\label{Eq:LB}
  \lim_{\b\to\infty} \liminf_{n\to\infty} \max_{\s \in \Sc}
  d_\s\left(\frac{3}{2}n\log n - \b n\right) = 1.
\end{equation}
\end{theorem}

\subsection{A connection to cryptography}

We mention another motivation for studying the product replacement
chain in the case $G=(\Z/q)^m$ for a prime $q \ge 2$ and integers $m \ge
1$. It comes from a public-key authentication protocol proposed by
Sotiraki \cite{Sotiraki}, which we now briefly describe. In the
protocol, a verifier wants to check the identity of a prover based on
the time needed to answer a challenge.

First, the prover runs the Markov chain with $G = (\Z/q)^m$ and $n = m$,
which can be interpreted as performing a random walk on $SL_n(\Z/q)$,
where $\s(k)$ is viewed as the $k$-th row of a $n \times n$
matrix. (In each step, a random row is either added to or subtracted
from another random row.)

After $t$ steps, the prover records the resulting matrix $A \in
SL_n(\Z/q)$ and makes it public. To authenticate, the verifier gives
the prover a vector $x \in (\Z/q)^n$ and challenges her to compute $y :=
Ax$. The prover can perform this calculation in $O(t)$ operations by
retracing the trajectory of the random walk.

Without knowing the trajectory, if $t$ is large enough, an adversary
will not be able to distinguish $A$ from a random matrix and will be
forced to perform the usual matrix-vector multiplication (using $n^2$
operations) to complete the challenge. Thus, the question is whether
$t \ll n^2$ is large enough for the matrix $A$ to become sufficiently
random, so that the prover can answer the challenge much faster than
an adversary.

Note that when $n > m$, the product replacement chain on $G = (\Z/q)^m$
amounts to the projection of the random walk on $SL_n(\Z/q)$ onto the
first $m$ columns. Thus, Theorem \ref{Thm:main} shows that when $m$ is
fixed and $n \rightarrow \infty$, the mixing time for the first $m$
columns is around $\frac{3}{2} n \log n$. One then hopes that the
mixing of several columns is enough to make it computationally
intractable to distinguish $A$ from a random matrix; this would
justify the authentication protocol, as $n \log n \ll n^2$.

We remark that when $t$ is much larger than the mixing time of the
random walk on $SL_n(\Z/q)$ generated by row and additions and
subtractions, it is information theoretically impossible for an
adversary to distinguish $A$ from a random matrix. However, the
diameter of the corresponding Cayley graph on $SL_n(\Z/q)$ is known to
be of order $\Theta\left( \frac{n^2}{\log_q n} \right)$ \cite{AHM,
  Christofides}, so a lower bound of the same order necessarily holds
for the mixing time. Diaconis and Saloff-Coste \cite[Section 4,
  p.\ 420]{DSC96} give an upper bound of $O(n^4)$, which was
subsequently improved to $O(n^3)$ by Kassabov \cite{Kassabov}. Closing
the gap between $n^3$ and $\frac{n^2}{\log n}$ remains an open
problem.

\subsection{Outline of proof}

The proof of Theorem \ref{Thm:main} analyzes the mixing behavior in
several stages:
\begin{itemize}
\item an initial ``burn-in'' period lasting around $n \log n$ steps,
  after which the group elements appearing in the configuration are
  not mostly confined to any proper subgroup of $G$;
\item an averaging period lasting around $\frac{1}{2} n \log n$ steps,
  after which the counts of group elements become close to their
  average value under the stationary distribution; and
\item a coupling period lasting $O(n)$ steps, after which our chain
  becomes exactly coupled to the stationary distribution with high
  probability.
\end{itemize}
The argument is in the spirit of \cite{BHP}, but a more elaborate
analysis is required in the second and third stages. To analyze the
first stage, for a fixed proper subgroup $H$, the number of group
elements in $H$ appearing in the configuration is a birth-and-death
process whose transition probabilities are easy to estimate. The
analysis of the resulting chain is the same as in \cite{BHP}, and we
can then union bound over all proper subgroups $H$.

In the second stage, for a given starting configuration $\s_0 \in
\Sc$, we consider quantities $n_{a,b}(\s)$ counting the number of
sites $k$ where $\s_0(k) = a$ and $\s(k) = b$. A key observation
(which also appears in \cite{BHP}) is that by symmetry, projecting the
Markov chain onto the values $(n_{a,b}(\s_t))_{a, b \in G}$ does
not affect the mixing behavior. Thus, it is enough to understand the
mixing behavior of the counts $n_{a,b}$.

One expects these counts to evolve towards their expected value
$\E_{\s \sim \pi} n_{a,b}(\s)$ as the chain mixes. To carry out the
analysis rigorously, we write down a stochastic difference equation
for the $n_{a,b}$ and analyze it via the Fourier
transform. Intuitively, as $n \rightarrow \infty$, the process
approaches a ``hydrodynamic limit'' so that it becomes approximately
deterministic. It turns out that after about $\frac{1}{2} n \log n$
steps, the $n_{a,b}$ are likely to be within $O(\sqrt{n})$ of their
expected value. Our analysis requires a sufficiently ``generic''
initial configuration, which is why the first stage is necessary.

Finally, in the last stage, we show that if the $(n_{a,b}(\s))_{a,b\in
  G}$ and $(n_{a,b}(\s'))_{a,b\in G}$ for two configurations are
within $O(\sqrt{n})$ in $\ell^1$ distance, they can be coupled to be
exactly the same with high probability after $O(n)$ steps of the
Markov chain. A standard argument involving coupling to the stationary
distribution then implies a bound on the mixing time.

The main idea to prove the coupling bound is that even if the $\ell^1$
distance evolves like an unbiased random walk, there is a good chance
that it will hit $0$ due to random fluctuations. A similar argument is
used to prove cutoff for lazy random walk on the hypercube \cite[Chapter
  18]{LevinPeresWilmer}. However, some careful accounting is necessary
in our setting to ensure that in fact the $\ell^1$ distance does not
increase in expectation and to ensure sufficient fluctuations.

\subsection{Organization of the paper}

The rest of the paper is organized as follows. In Section
\ref{Sec:UB}, we state (without proof) the key lemmas describing the
behavior in each of the three stages and use these to prove the upper
bound \eqref{Eq:UB} in Theorem \ref{Thm:main}. Sections
\ref{Sec:DE-Proofs} and \ref{Sec:Coupling-Proof} contain the proofs of
these lemmas. Finally, in Section \ref{Sec:LB}, we prove the lower
bound \eqref{Eq:LB} in Theorem \ref{Thm:main}; this is mostly a matter
of verifying that the estimates used in the upper bound were tight.

\subsection{Notation}

Throughout this paper, we use $c, C, C', \dots$, to denote
absolute constants whose exact values may change from line to line,
and also use them with subscripts, for instance, $C_G$ to specify its
dependency only on $G$. We also use subscripts with big-$O$ notation,
e.g.\ we write $O_G(\,\cdot\,)$ when the implied constant depends only
on $G$.

\section{Proof of Theorem \ref{Thm:main} (\ref{Eq:UB})}\label{Sec:UB}

Let us fix a finite group $G$ and denote its cardinality by $\Qc :=
|G|$. For a configuration $\s \in \Sc$, let $n_a(\s)$ denote the
number of sites having group element $a$, i.e.,
\[ n_a(\s) := |\{i \in [n] \ : \ \s(i)=a\}|. \]

\subsection{The burn-in period}

For a proper subgroup $H \subseteq G$, let
\[ n_{non}^H(\s) := \sum_{a \in G \setminus H} n_a(\s) \]
denote the number of sites not in $H$, and define for $c \in (0, 1)$
the set
\[ \Snon{c} := \{\s \in \Sc \ : \ n_{non}^H(\s) \ge cn \text{ for all proper subgroups $H \subseteq G$} \}. \]
Thus, $\Snon{c}$ is the set of states $\s$ where the group elements
appearing in $\s$ are not mostly confined to any particular proper
subgroup of $G$. The next lemma shows that we reach $\Snon{1/3}$
in about $n \log n$ steps, and once we reach $\Snon{1/3}$, we
remain in $\Snon{1/6}$ for $n^2$ steps with high probability. Note
that $n^2$ is much larger than the overall mixing time, so we may
essentially assume that we are in $\Snon{1/6}$ for all of the
later stages.

\begin{lemma}\label{Lem:Burn}
  Let $\t_{1/3} := \min\{t \ge 0 : \s_t \in \Snon{1/3}\}$ be the
  first time to hit $\Snon{1/3}$. Then for all large enough $n$
  and for all large enough $\b > 0$,
  \[ \max_{\s \in \Sc} \Pr_\s(\t_{1/3} > n \log n + \b n) \le \frac{120 \Qc}{\b^2}. \]
  Moreover, there exists a constant $C_G$ depending only on $G$ such that
  \[ \max_{\s \in \Snon{1/3}}\Pr_\s \left(\s_t \notin \Snon{1/6} \ \text{\rm for some $t \le n^2$}\right) \le C_G n^2e^{-n/10}. \]
\end{lemma}

\begin{proof}
  Fix a proper subgroup $H \subset G$, and consider what happens to
  $n_{non}^H(\s_t)$ at time $t$. Suppose our next step is to replace
  $\s(i)$ with $\s(i)\s(j)$.

  If $\s(j) \in H$, then $n_{non}^H(\s_{t+1}) = n_{non}^H(\s_t)$. If
  $\s(j) \not\in H$ and $\s(i) \in H$, then $n_{non}^H(\s_{t+1}) =
  n_{non}^H(\s_t) - 1$. Finally, if $\s(j), \s(i) \not\in H$, then
  $\s(i)\s(j)$ may or may not be in $H$, so $n_{non}^H(\s_{t+1}) \ge
  n_{non}^H(\s_t) - 1$.

  Let $(N_t)_{t \ge 0}$ be the birth-and-death chain with the
  following transition probabilities for $1 \le k \le n$:
  \begin{align*}
    \Pr(N_{t+1} = k+1 \mid N_t = k) &= \frac{k(n-k)}{n(n-1)} \\
    \Pr(N_{t+1} = k-1 \mid N_t = k) &= \frac{k(k-1)}{n(n-1)} \\
    \Pr(N_{t+1} = k \mid N_t = k) &= \frac{n-k}{n}.
  \end{align*}
  We start this chain at $N_0 = n^H_{non}(\s_0)$; note that because
  the elements appearing in $\s_0$ generate $G$, we are guaranteed
  to have $n^H_{non}(\s_0) > 0$.

  The above birth-and-death chain corresponds to the behavior of
  $(n^H_{non}(\s_t))$ if whenever $\s(j), \s(i) \not\in H$, it always
  happened that $\s(i)\s(j) \in H$. Thus, $(n^H_{non}(\s_t))$
  stochastically dominates $(N_t)$.

  The chain $(N_t)$ is precisely what is analyzed in \cite{BHP} for
  the case $G = \Z/2$. Let
  \[T_k := \min\{t \ge 0 : N_t=k\}.\]
  Then, we have $\E_{k-1}T_k \le \frac{n^2}{k(n-2k)}$ \cite[(2) in the
    proof of Lemma 1]{BHP} and thus $\E_1 (T_{n/3})
  =\sum_{k=2}^{n/3}\E_{k-1}T_k \le n \log n + n$.  On the other hand,
  setting $v_k=\Var_{k-1}(T_k)$, we have $v_2 \le n^2$,
  \[v_{k+1}\le \frac{k}{n-k}v_k + \frac{54 n^2}{k^2},\]
  and $\Var_1 (T_{n/3}) = \sum_{k=2}^{n/3}v_k \le 110 n^2$ \cite[The proof of Lemma 1]{BHP}.
  Hence by Chebyshev's inequality for
  all large enough $\b > 0$,
  \[\Pr_1(T_{n/3} > n \log n + \b n) \le \frac{120}{\b^2}.\]
  Moreover, we have $\Pr_{n/3} \left( T_{n/6} \le n^2 \right) \le n^2e^{-n/10}$.
  Indeed, this follows from the fact that for $m<k$, we have
  \[\Pr_k(T_m \le n^2) \le n^2\frac{\pi_{\rm BD}(m)}{\pi_{\rm BD}(k)},\]
  where $\pi_{\rm BD}(k)={n \choose k}/(2^n-1)$ \cite[(5) and the
    following in the proof of Proposition 2]{BHP}.

    We now take a union bound over all the
  proper subgroups $H$.
\end{proof}

\subsection{The averaging period}

In the next stage, the counts $n_a(\s_t)$ go toward their average
value. We actually analyze this stage in two substages, looking at a
``proportion vector'' and ``proportion matrix'', as described below.

\subsubsection{Proportion vector chain}

For a configuration $\s \in \Sc$, we consider the $\Qc$-dimensional
vector $(n_a(\s)/n)_{a \in G}$, which we call the {\it proportion
  vector} of $\s$. One may check that for a typical $\s \in \Sc$, each
$n_a(\s)/n$ is about $1/\Qc$. For each $\d > 0$, we define the
$\d$-{\it typical set}
\[ \Sc_\ast(\d) := \left\{\s \in \Sc \ : \ \left\|\left(\frac{n_a(\s)}{n}\right)_{a \in G} - \left(\frac{1}{\Qc}\right)_{a \in G}\right\| \le \d \right\}, \]
where $\| \cdot \|$ denotes the $\ell^2$-norm in $\R^G$.

The following lemma implies that starting from $\s \in \Snon{1/3}$, we
reach $\Sast{\d}$ in $O_\d(n)$ steps with high
probability. The proof is given in Section \ref{Subsec:vector}.

\begin{lemma}\label{Lem:1stDE}
  Consider any $\s \in \Snon{1/3}$ and any constant $\d>0$. There exists a constant $C_{G, \d}$
  depending only on $G$ and $\d$ such that for any $T \ge C_{G, \d} n$, we have
  \[ \Pr_\s\left(\s_T \notin \Sast{\d} \right) \le \frac{1}{n} \]
  for all large enough $n$.
\end{lemma}

\subsubsection{Proportion matrix chain}

We actually need a more precise averaging than what is provided by
Lemma \ref{Lem:1stDE}. Fix a configuration $\s_0 \in \Sc$. For any $\s
\in \Sc$ and for any $a, b \in G$, define
\[ n_{a,b}^{\s_0}(\s) := |\{i \in [n] \ : \ \s_0(i)=a, \s(i)=b \}|. \]
If we run the Markov chain $(\s_t)_{t\ge 0}$ with initial state
$\s_0$, then $n_{a,b}^{\s_0}(\s_t)$ is the number of sites that
originally contained the element $a$ (at time 0) but now contain $b$
(at time $t$). Note that
\[ \sum_{b \in G} n_{a,b}^{\s_0}(\s) = n_a(\s_0) \quad \text{and} \quad \sum_{a \in G} n_{a,b}^{\s_0}(\s) = n_b(\s).\]

We can then associate with $(\s_t)_{t \ge 0}$ another Markov chain
$\left(n_{a,b}^{\s_0}(\s_t)/n_a(\s_0)\right)_{a, b \in G}$ for $t \ge
0$, which we call the {\it proportion matrix chain} ({\it with respect to
  $\s_0$}). The state space for the proportion matrix chain is $\{0,
1, \dots, n\}^{G \times G}$, and the transition probabilities depend
on $\s_0$.

The proportion matrix acts like a ``sufficient statistic'' for
analyzing our Markov chain started at $\s_\ast$, because of the
permutation invariance of our dynamics. In fact, as the following
lemma shows, the distance to stationarity of the proportion matrix
chain is equal to the distance to stationarity of the original chain.

\begin{lemma}\label{Lem:nij}
  Let $\s_\ast \in \Sc$ be a configuration. For the Markov chain
  $(\s_t)_{t \ge 0}$ with initial state $\s_\ast$, we consider
  $\left(n_{a, b}^{\s_\ast}(\s_t)\right)_{a, b \in G}$. Let
  $\overline{\pi}^{\s_\ast}$ be the stationary measure for the Markov
  chain $\{(n_{a, b}^{\s_\ast}(\s_t))_{a, b \in G}\}_{t \ge 0}$ on
  $\left\{0, 1, \dots, n\right\}^{G \times G}$. Then, for every $t \ge
  0$, we have
  \[ \normTV{ \Pr_{\s_\ast}(\s_t \in \cdot \ ) - \pi } = \normTV{ \Pr_{\s_\ast}\left( (n_{a, b}^{\s_\ast}(\s_t))_{a, b \in G} \in \cdot \ \right) - \overline{\pi}^{\s_\ast} }. \]
\end{lemma}
\begin{proof}
  For any matrix $N = (N_{a,b})_{a,b \in G} \in \{0, 1, \ldots ,
  n\}^{G \times G}$, write
  \[ \Xc_{(N)} := \left\{\s \in \Sc \ : \ n_{a, b}^{\s_\ast}(\s)=N_{a, b} \ \text{for all $a, b \in G$}\right\} \]
  for the set of configurations with $N$ as their proportion matrix.

  Since the distribution of $\s_t$ is invariant under permutations on
  sites $i \in [n]$ preserving the set $\{ i : \s_\ast(i) = a\}$ for
  every $a \in G$, the conditional probability measures
  $\Pr_{\s_\ast}\left(\s_t \in \cdot \mid \s_t \in \Xc_{(N)} \right)$ and
  $\pi( \ \cdot \mid \Xc_{(N)})$ are both uniform on $\Xc_{(N)}$.
  This implies that for each $\s \in \Xc_{(N)}$,
  \[ |\Pr_{\s_\ast}(\s_t =\s) - \pi(\s)| = \frac{1}{\left|\Xc_{(N)}\right|}\left|\Pr_{\s_\ast}\left((n^{\s_\ast}_{a, b}(\s_t))_{a, b \in G}=N\right)- \overline{\pi}^{\s_\ast}(N) \right|, \]
  and summing over all $\s \in \Xc_{(N)}$ and all $N$, we obtain the
  claim.
\end{proof}

For $\s_0 \in \Sc$ and $r > 0$, define the set of configurations
\[ \Sast{\s_0, r} := \left\{ \s \in \Sc \ : \ \left\| \left(\frac{n^{\s_0}_{a, b}(\s)}{n_a(\s_0)}\right)_{b \in G} - \left(\frac{1}{\Qc}\right)_{b \in G}\right\| \le r \text{ for all $a
  \in G$} \right\}. \]
Roughly speaking, the following lemma shows that starting from a
typical configuration $\s_\ast \in \Sast{\frac{1}{4\Qc}}$, we need
about $\frac{1}{2}n \log n$ steps to reach $\Sast{\s_\ast,
  \frac{R}{\sqrt{n}}}$, where $R$ is a constant. We show this fact in
a slightly more general form where the initial state need not be
$\s_\ast$; the proof is given in Section \ref{Subsec:matrix}.

\begin{lemma}\label{Lem:2ndDE}
  Consider any $\s_\ast, \s'_\ast \in \Sast{\frac{1}{4\Qc}}$, and let
  $T := \ceil{\frac{1}{2} n \log n}$. There exists a constant $C_G >
  0$ depending only on $G$ such that for any given $R > 0$, we have
  \[ \Pr_{\s'_\ast}\left(\s_T \notin \Sast{\s_\ast, \frac{R}{\sqrt{n}}}\right) \le C_G e^{-R} + \frac{1}{n} \]
  for all large enough $n$.
\end{lemma}

\subsection{The coupling period}

After reaching $\Sast{\s_\ast, \frac{R}{\sqrt{n}}}$, we show that only
$O(n)$ additional steps are needed to mix in total variation
distance. The main ingredient in the proof is a coupling of proportion
matrix chains so that they coalesce in $O(n)$ steps when they both
start from configurations $\s, \tilde\s \in \Sast{\s_\ast,
  \frac{R}{\sqrt{n}}}$. We construct such a coupling and prove the
following lemma in Section \ref{Sec:Coupling-Proof}.

\begin{lemma}\label{Lem:RW}
  Consider any $\s_\ast \in \Sast{\frac{1}{5\Qc^3}}$, and let $R > 0$. Suppose $\s, \tilde
  \s \in \Sast{\s_\ast, \frac{R}{\sqrt{n}}}$. Then, there exists a
  coupling $(\s_t, \tilde \s_t)$ of the Markov chains with initial
  states $(\s, \tilde \s)$ such that for a given $\b > 0$ and all
  large enough $n$,
  \[ \Pr_{\s, \tilde \s}(\t > \b n) \le \frac{32\Qc^2 R}{\sqrt{\b}}, \]
  where $\t:=\min\{t \ge 0 : n^{\s_\ast}_{a, b}(\s_t) =
  n^{\s_\ast}_{a, b}(\tilde \s_t) \ \text{\rm for all $a, b \in G$}\}$.
\end{lemma}

To translate this coupling time into a bound on total variation
distance, we need also the simple observation that the stationary
measure $\pi$ concentrates on $\Sast{\s_\ast, \frac{R}{\sqrt{n}}}$
except for probability $O(1/R^2)$, as given in the next lemma.

\begin{lemma}\label{Lem:stationary}
  For the stationary distribution $\pi$ of the chain $(\s_t)_{t \ge 0}$,
  for every $R>0$ and for all $n > m$,
  \[ \pi\left(\s \notin \Sast{\frac{R}{\sqrt{n}}}\right) \le \frac{C_G}{R^2}. \]
  Moreover for every $\d<1/(2\Qc)$, for every $R>0$ and for all $n>m$,
  \[ \max_{\s_\ast \in \Sast{\d}}\pi\left(\s \notin \Sast{\s_\ast, \frac{R}{\sqrt{n}}}\right) \le \frac{2 C_G \Qc}{R^2}, \]
  where $C_G$ and $m$ are constants depending only on $G$.
\end{lemma}

\begin{proof}
  Observe that since the stationary distribution $\pi$ is uniform on
  $\Sc$, it is given by the uniform distribution ${\rm Unif}$ on $G^n$
  conditioned on $\Sc$. Note that we can always generate $G$ using
  each of its $|G|$ elements, so we have an easy lower bound of $|\Sc|
  \ge |G|^{n - |G|}$. Consequently, we have
  \begin{align*}
  \pi\left(\s \notin \Sast{\frac{R}{\sqrt{n}}}\right)
  &\le |G|^{|G|} {\rm Unif}\left(\s \notin \Sast{\frac{R}{\sqrt{n}}}\right) \\
  &\le |G|^{|G|} \sum_{a \in G}{\rm Unif}\left(\left| \frac{n_a(\s)}{n}-\frac{1}{\Qc}\right| \ge \frac{R}{\sqrt{n}}\right) \le \frac{|G|^{|G|}}{R^2}\left(1-\frac{1}{\Qc}\right).
  \end{align*}
  Concerning the second assertion, we note that $n_a(\s_\ast) \ge
  (1/\Qc -\d)n$ for each $a \in G$; the rest follows similarly, so we
  omit the details.
\end{proof}

\begin{remark}\label{Rem:1}
  In Lemma \ref{Lem:stationary} above, we have given a very loose
  bound on $C_G$ for sake of simplicity. Actually, it is not hard to
  see that holding $G$ fixed, we have $\lim_{n\rightarrow \infty}
  |\Sc|/|G|^n = 1$. See also \cite[Section 6.B.]{DSC98} for more
  explicit bounds for various families of groups.
\end{remark}

Together, Lemmas \ref{Lem:2ndDE}, \ref{Lem:RW}, and \ref{Lem:stationary}
imply the following bound for total variation distance.

\begin{lemma} \label{Lem:coupling-tv}
  Let $\b > 0$ be given, and let $T := \ceil{\frac{1}{2} n \log n} +
  \ceil{\b n}$. Then, for any $\s_\ast \in \Sast{\frac{1}{5\Qc^3}}$,
  we have
  \[ \normTV{ \Pr_{\s_\ast}(\s_T \in \cdot \ ) - \pi } \le \frac{C_G}{\b^{1/4}}, \]
  where $C_G$ is a constant depending only on $G$.
\end{lemma}
\begin{proof}
  Let $\tilde\s$ be drawn from the stationary distribution
  $\pi$. Define
  \[ \tau = \min \left\{ t \ge 0 : n^{\s_\ast}_{a,b}(\s_t) = n^{\s_\ast}_{a,b}(\tilde\s_t) \text{ for all $a, b \in G$}\right\}, \]
  where $(\tilde\s_t)$ is a Markov chain started at $\tilde\s$. Let
  $\overline{\pi}^{\s_\ast}$ denote the stationary distribution for
  the proportion matrix with respect to $\s_\ast$. Since $\tilde\s$
  was drawn from $\pi$, the proportion matrix of $\tilde\s_t$ remains
  distributed as $\overline{\pi}^{\s_\ast}$ for all $t$.

  We first run $\s$ and $\tilde\s$ independently up until time $T_1 :=
  \ceil{\frac{1}{2} n \log n}$. For a parameter $R$ to be specified
  later, consider the events
  \[ \Gc := \left\{ \s_{T_1} \in \Sast{\s_\ast, \frac{R}{\sqrt{n}}} \right\}, \qquad \tilde{\Gc} := \left\{ \tilde\s_{T_1} \in \Sast{\s_\ast, \frac{R}{\sqrt{n}}} \right\}. \]
  Lemma \ref{Lem:2ndDE} implies that $\Pr(\Gc^{\sf c}) \le
  C_G e^{-R} + \frac{1}{n}$, and Lemma \ref{Lem:stationary} implies that
  $\Pr(\tilde\Gc^{\sf c}) \le \frac{2 C_G \Qc}{R^2}$.

  Let $T_2 := \ceil{\b n}$. Starting from time $T_1$, as long as both
  $\Gc$ and $\tilde\Gc$ hold, we may use Lemma \ref{Lem:RW} to form a
  coupling $(\s_t, \tilde\s_t)$ so that
  \[ \Pr_{\s_\ast, \s_\ast} \Big( n^{\s_\ast}_{a, b}(\s_{T_1+T_2}) \ne n^{\s_\ast}_{a, b}(\tilde\s_{T_1+T_2}) \text{ for some $a, b \in G$} \,\Big|\, \Gc \cap \tilde\Gc \Big) \le \frac{C\Qc^2 R}{\sqrt{\b}}. \]
  Setting $R = \b^{1/4}$, we conclude that
  \begin{align*}
    &\Pr_{\s_\ast, \s_\ast} \Big( n^{\s_\ast}_{a, b}(\s_{T_1+T_2}) \ne n^{\s_\ast}_{a, b}(\tilde\s_{T_1+T_2}) \text{ for some $a, b \in G$} \Big) \le \frac{C\Qc^2 R}{\sqrt{\b}} + \Pr(\Gc^{\sf c}) + \Pr(\tilde\Gc^{\sf c}) \\
    &\qquad\le \frac{C\Qc^2 R}{\sqrt{\b}} + \left(C_G e^{-R} + \frac{1}{n}\right) + \frac{2C_G \Qc}{R^2} = O_G\left(\frac{1}{\b^{1/4}}\right).
  \end{align*}
  We have $T = T_1 + T_2$, and recall that the proportion matrix for
  $\tilde\s$ is stationary for all time. This yields
  \[ \normTV{ \Pr_{\s_\ast}\left( (n^{\s_\ast}_{a, b}(\s_T))_{a, b \in G} \in \cdot \ \right) - \overline{\pi}^{\s_\ast} } = O_G\left(\frac{1}{\b^{1/4}}\right). \]
  The result then follows by Lemma \ref{Lem:nij}.
\end{proof}

\subsection{Proof of the main theorem}

We now combine the lemmas from the burn-in, averaging, and coupling
periods to complete the proof of the upper bound in Theorem
\ref{Thm:main}.

\begin{proof}[Proof of Theorem \ref{Thm:main} (\ref{Eq:UB})]
  Define $T_1 := \ceil{n \log n + \b n}$, $T_2 := \ceil{\b n}$, and
  $T_3 := \ceil{\frac{1}{2} n \log n} + \ceil{\b n}$.

  Let $\tau_{1/3}$ be the first time to hit $\Snon{1/3}$ as in
  Lemma \ref{Lem:Burn}. Then, Lemma \ref{Lem:Burn} implies that for
  any $\s_1 \in \Sc$ and any $t \ge 0$, we have
  \begin{align}
    d_{\s_1}(T_1 + t) &\le \Pr_{\s_1}\left( \tau_{1/3} > T_1 \right) + \max_{\s \in \Snon{1/3}} d_\s(t) \nonumber \\
    &\le \frac{120 \Qc}{\b^2} + \max_{\s \in \Snon{1/3}} d_\s(t). \label{Eq:stage1}
  \end{align}
  Next, by Lemma \ref{Lem:1stDE}, for any $\s_2 \in \Snon{1/3}$
  and when $\b$ and $n$ are sufficiently large, we have that
  $\Pr_{\s_2} \left( \s_{T_2} \not\in \Sast{\frac{1}{5\Qc^3}} \right)
  \le \frac{1}{n}$. Consequently, for $\s_2 \in \Snon{1/3}$, we
  have
  \begin{equation} \label{Eq:stage2}
    d_{\s_2}(T_2 + t) \le \frac{1}{n} + \max_{\s_\ast \in \Sast{\frac{1}{5\Qc^3}}} d_{\s_\ast}(t).
  \end{equation}
  Finally, Lemma \ref{Lem:coupling-tv} states that
  \begin{equation} \label{Eq:stage3}
    \max_{\s_\ast \in \Sast{\frac{1}{5\Qc^3}}} d_{\s_\ast}(T_3) \le \frac{C_G}{\b^{1/4}}.
  \end{equation}

  Thus, combining \eqref{Eq:stage1}, \eqref{Eq:stage2}, and
  \eqref{Eq:stage3}, we obtain for any $\s \in \Sc$ that
  \begin{align*}
    d_\s\left(\frac{3}{2} n \log n + 4 \b n \right) &\le d_\s\left(T_1 + T_2 + T_3 \right) \nonumber \\
    &\le \frac{120 \Qc}{\b^2} + \frac{1}{n} + \frac{C_G}{\b^{1/4}}
  \end{align*}
  sending $n \rightarrow \infty$ and then $\b \rightarrow \infty$
  yields \eqref{Eq:UB}.
\end{proof}

\section{Proofs for the averaging period} \label{Sec:DE-Proofs}

In this section, we prove Lemmas \ref{Lem:1stDE} and
\ref{Lem:2ndDE}. The proofs are based on analyzing stochastic
difference equations satisfied by the Fourier transform of the
proportion vector or matrix.

\subsection{The Fourier transform for $G$}

We first establish some notation and preliminaries for the Fourier
transform. Let $G^\ast$ be a complete set of non-trivial irreducible
representations of $G$. In other words, for each $\r \in G^\ast$, we
have a finite dimensional complex vector space $V_\r$ such that $\r: G
\to GL(V_\r)$ is a non-trivial irreducible representation, and any
non-trivial irreducible representation of $G$ is isomorphic to some unique $\r \in
G^\ast$. Moreover, we may equip each $V_\r$ with an inner product for
which $\r \in G^\ast$ is unitary.

For a configuration $\s \in \Sc$ and for each $\r \in G^\ast$, we
consider the matrix acting on $V_\r$ given by
\[ x_\r(\s) := \sum_{a \in G} \frac{n_a(\s)}{n}\r(a), \]
so that $x_\r(\s)$ is the Fourier transform of the proportion vector
at the representation $\r$. We write $x(\s):=(x_\r(\s))_{\r \in
  G^\ast}$.

Let $\wt{V} := \bigoplus_{\r \in G^\ast}\End_\C(V_\r)$, and write
$d_\r := \dim_\C V_\r$. For an element $x = (x_\r)_{\r \in G^\ast} \in
\wt{V}$, we define a norm $\| \cdot \|_{\wt{V}}$ given by
\begin{equation*}\label{Eq:norm}
\|x\|_{\wt{V}}^2 := \frac{1}{\Qc}\sum_{\r \in G^\ast}d_\r \|x_\r\|_{\HS}^2,
\end{equation*}
where $\langle A, B \rangle_{\HS} =\Tr(A^\ast B)$ denotes the
Hilbert-Schmidt inner product in $\End_\C(V_\r)$ and $\|\cdot\|_{\HS}$
denotes the corresponding norm. (Note that $\langle \cdot, \cdot
\rangle_{\HS}$ and $\|\cdot\|_{\HS}$ depend on $\r$, but for sake of
brevity, we omit the $\r$ when there is no danger of confusion.)

The Peter-Weyl theorem \cite[Chapter 2]{DiaconisRep} says that
\begin{equation*} \label{Eq:Peter-Weyl}
  L^2(G) \cong \C \oplus \wt{V},
\end{equation*}
where the isomorphism is given by the Fourier transform. The
Plancherel formula then reads
\begin{equation}\label{Eq:Plancherel}
\|x(\s)\|_{\wt{V}}^2 = \left\|\left(\frac{n_a(\s)}{n}\right)_{a \in G} - \left(\frac{1}{\Qc}\right)_{a \in G}\right\|^2.
\end{equation}
Thus, in order to show that $\s \in \Sast{\d}$, it
suffices to show that $\|x(\s)\|_{\wt{V}}$ is small. A similar
argument may be applied to the proportion matrix instead of the
proportion vector.

Finally, for an element $A \in \End_\C(V_\r)$, we will at times also
consider the \emph{operator norm} $\|A\|_{op} := \sup_{v \in V_\r, v \neq 0}
\|Av\| / \|v\|$. We will also sometimes use the following (equivalent)
variational characterization of the operator norm:
\begin{align*}
  \sup_{\substack{X \in \End_\C(V_\r) \\ \|X\|_{\HS} = 1}} \|XA\|^2_{\HS} &= \sup_{\substack{X \in \End_\C(V_\r) \\ \|X\|_{\HS} = 1}} \Tr (XAA^*X^*) = \sup_{\substack{X \in \End_\C(V_\r) \\ \|X\|_{\HS} = 1}} \Tr (X^*XAA^*) \\
  &= \sup_{\substack{Y \in \End_\C(V_\r) \\ Y = Y^*, \;\; \Tr Y = 1}} \langle Y , AA^* \rangle_{\HS} = \|AA^*\|_{op} = \|A\|_{op}^2.
\end{align*}

\subsubsection{The special case of $G = \Z/q$}

On a first reading of this section, the reader may wish to consider
everything for the special case of $G = \Z/q$ for some integer $q \ge 2$. In
that case, each representation is one-dimensional, and the
representations can be indexed by $\ell = 0, 1, 2, \ldots , q - 1$. The
Fourier transform is then particularly simple: the coefficients are
scalar values
\[ x_\ell(\s) = \sum_{a = 0}^{q - 1} \frac{n_a(\s)}{n} \o^{a \ell}, \]
where $\o := e^{\frac{2\pi i}{q}}$ is a primitive $q$-th root of
unity.

This special case already illustrates most of the main ideas while
simplifying the estimates in some places (e.g.\ matrix inequalities we
use will often be immediately obvious for scalars).

\subsection{A stochastic difference equation for the $n_a$}

For $a \in G$, we next analyze the behavior of $n_a(\s_t)$ over
time. For convenience, we write $n_a(t) = n_a(\s_t)$. Let $\Fc_t$
denote the $\s$-field generated by the Markov chain $(\s_t)_{t \ge 0}$
up to time $t$. Then, our dynamics satisfy the equation
\begin{equation}\label{Eq:DE0}
  \E[n_a(t+1)-n_a(t) \mid \Fc_t] = \sum_{b \in G} \frac{n_{ab^{-1}}(t) n_b(t) }{2n(n-1)}+\sum_{b \in G} \frac{n_{ab}(t) n_b(t)}{2n(n-1)} - \frac{n_a(t)}{n}.
\end{equation}
Note that $|n_a(t + 1) - n_a(t)| \le 1$ almost surely. Thus, for each
$a \in G$, we can write the above as a stochastic difference equation
\begin{equation}\label{Eq:DE}
  n_a(t+1) - n_a(t) = \sum_{b \in G} \frac{n_{ab^{-1}}(t) n_b(t) }{2n(n-1)}+\sum_{b \in G} \frac{n_{ab}(t) n_b(t)}{2n(n-1)} - \frac{n_a(t)}{n} + M_a(t+1),
\end{equation}
where $\E[M_a(t+1) \mid \Fc_t] = 0$ and $|M_a(t)| \le 2$ almost
surely.

It is easiest to analyze this equation through the Fourier
transform. Writing $x_\r(t) = x_\r(\s_t)$, we calculate from \eqref{Eq:DE} that
\[ x_\r(t+1) - x_\r(t) = \frac{1}{n-1}x_\r(t) \left(\frac{x_\r(t) + x_\r(t)^\ast}{2} -\frac{n-1}{n}\right) + \wh{M}_\r(t+1), \]
where $\wh{M}_\r(t) := \frac{1}{n}\sum_{a \in G}M_a(t) \r(a)$.
For convenience, write
\[X_\r(t) = \frac{1}{n - 1}\left(\frac{x_\r(t) + x_\r(t)^\ast}{2} - \frac{n-1}{n}\right),\]
so that our equation becomes
\begin{equation} \label{Eq:DEfourier}
  x_\r(t+1) - x_\r(t) = x_\r(t) X_\r(t) + \wh{M}_\r(t+1).
\end{equation}
Note that we have
\[ \|x_\r(t)\|_{\HS} \le \sqrt{d_\r}, \qquad \E[\wh{M}_\r(t+1) \mid \Fc_t] = 0, \qquad\text{and}\qquad \|\wh{M}_\r(t)\|_{\HS} \le \frac{2 \Qc\sqrt{d_\r}}{n}, \]
and thus,
\[ \|x(t)\|_{\wt{V}} \le 1 \qquad\text{and}\qquad \|\wh{M}(t)\|_{\wt{V}}\le \frac{2\Qc}{n},\]
where $\wh{M}=(\wh{M}_\r)_{\r \in G^\ast}$.

\subsection{A general estimate for stochastic difference equations}

Before proving Lemma \ref{Lem:1stDE}, we also need a technical lemma
for controlling the behavior of stochastic difference equations, which
will be used to analyze \eqref{Eq:DEfourier} as well as other similar
equations.

\begin{lemma} \label{Lem:GenDE}
  Let $(z(t))_{t \ge 0}$ be a sequence of $[0,1]$-valued random
  variables adapted to a filtration $(\Fc_t)_{t \ge 0}$. Let $\e \in
  (0, 1)$ be a small constant, and let $\f : \R^+ \to (0,1]$ be a
non-decreasing function.

  Suppose that there are $\Fc_t$-measurable random variables $M(t)$
  for which
  \begin{equation} \label{Eq:zDE}
    z(t+1) - z(t) \le -\e\f(t+1) z(t) + M(t+1)
  \end{equation}
  and which, for some constant $D$, satisfy the bounds
  \[ \E[ M(t+1) \mid \Fc_t ] \le D\e\sqrt{\e}, \qquad |M(t)| \le D\e. \]
  Then, for each $t$ and each $\lambda > 0$, we have
  \[ \Pr\left( z(t) \ge \lambda \sqrt{\e} + e^{- \e \int_0^t \f(s)\,ds} \cdot z(0) \right) \le C_{D,\f} e^{-c_{D,\f} \lambda^2} \]
  for constants $c_{D,\f}, C_{D,\f}$ depending only on $D$ and $\f$.
\end{lemma}
\begin{proof}
  Let us define for integers $t \ge 1$,
  \[ \F(t) := \e^{-1} \sum_{k = 1}^t \log \frac{1}{1-\e \f(k)}, \qquad \text{and} \qquad \F(0) := 0.\]
  Taking conditional expectations in
  the inequality relating $z(t+1)$ to $z(t)$, we have
  \[ \E[ z(t+1) \mid \Fc_t ] \le (1 - \e \f(t+1)) z(t) + D \e\sqrt{\e}. \]
  Rearranging and using the fact that $\f(t)$ is non-decreasing, we have
  \begin{align*}
    \E[ z(t+1) \mid \Fc_t ] - \frac{D\sqrt{\e}}{\f(0)} &\le (1 - \e \f(t+1)) z(t) - \frac{D\sqrt{\e}(1 - \e \f(t+1))}{\f(0)} \\
    &\le (1 - \e \f(t+1)) \left( z(t) - \frac{D\sqrt{\e}}{\f(0)} \right).
  \end{align*}
  Consequently,
  \[ Z_t := e^{\e \F(t)} \left( z(t) - \frac{D\sqrt{\e}}{\f(0)} \right) \]
  is a supermartingale, and its increments are bounded by
  \begin{equation} \label{Eq:Z-increment-bound}
    |Z_{t+1}-Z_t| \le e^{\e \F(t+1)}\left(|M(t+1)|+D \e\right) \le 2D\e e^{\e\F(t+1)}.
  \end{equation}
  Recall that $\f$ is non-decreasing, so that for all $t \ge s \ge 0$,
  we have
  \[ \F(t) = \F(s) + \e^{-1} \sum_{k = s + 1}^t \log \frac{1}{1 - \e \f(k)} \ge \F(s) + (t - s) \f(0). \]
  Using this with \eqref{Eq:Z-increment-bound}, we see that the sum of
  the squares of the first $t$ increments is at most
  \begin{align*}
    \sum_{s = 1}^{t} 4D^2 \e^2 e^{2\e\F(s)} &\le 4D^2\e^2 \sum_{s = 1}^t e^{2\e \F(t) - 2\e\f(0)(t - s)} \le 4D^2\e^2 e^{2\e\F(t)} \cdot \frac{1}{1 - e^{-2\e\f(0)}} \\
    &\le 4D^2\e^2 e^{2\e\F(t)} \cdot \frac{1}{1 - (1 - \frac{1}{2}\e\f(0))} = \frac{8D^2 \e}{\f(0)} \cdot e^{2\e\F(t)}.
  \end{align*}
  By the Azuma-Hoeffding inequality, this yields
  \[ \Pr\left(Z_t \ge \lambda \sqrt{\e} e^{\e\F(t)} + Z_0 \right) \le \exp\left( - \frac{\f(0) \lambda^2 \e \cdot e^{2\e\F(t)}}{16D^2 \e \cdot e^{2\e\F(t)}} \right) = \exp\left( -\frac{\f(0) \lambda^2}{16D^2} \right), \]
  which in turn implies
  \[ \Pr\left( z(t) \ge \frac{D\sqrt{\e}}{\f(0)} + e^{-\e\F(t)} z(0) + \lambda \sqrt{\e} \right) \le \exp\left( -\frac{\f(0)\lambda^2}{16D^2} \right). \]
  Finally, observe that $\F(t) \ge \sum_{k = 1}^t \f(k) \ge \int_0^t
  \f(s)\, ds$. The result then follows upon shifting and rescaling of
  $\lambda$.
\end{proof}

\subsection{Proportion vector chain: Proof of Lemma \ref{Lem:1stDE}}\label{Subsec:vector}

We first prove a bound for the Fourier coefficients $x_\r(t)$.

\begin{lemma} \label{Lem:1stDE-fourier}
  Consider any $\s \in \Snon{1/3}$ and any $\r \in G^\ast$. We have a
  constant $c_G$ depending only on $G$ for which
  \[ \Pr_\s\left( \bigcup_{t = 1}^{n^2} \left\{ \|x_\r(t)\|_{\HS} \ge \frac{1}{n^{1/8}} + e^{-c_G t/n}\cdot \|x_\r(0)\|_{\HS} \right\} \right) \le \frac{1}{n^3}. \]
  for all large enough $n$.
\end{lemma}

This immediately implies Lemma \ref{Lem:1stDE}.

\begin{proof}[Proof of Lemma \ref{Lem:1stDE}]
  With $c_G$ defined as in Lemma \ref{Lem:1stDE-fourier}, take $C_{G, \d}$
  large enough so that for any $T \ge C_{G, \d} n$,
  \[ \frac{1}{n^{1/8}} + e^{-c_G T/n}\sqrt{d_\r} \le \d. \]
  Then, Lemma \ref{Lem:1stDE-fourier} and Plancherel's formula yield
  \begin{align*}
    \Pr_\s\left(\s_T \notin \Sast{\d} \right) &\le \Pr_\s\left( \|x_\r(T)\|_{\HS} \ge \d \text{ for some $\r \in G^\ast$} \right) \\
    &\le \frac{\Qc}{n^3} \le \frac{1}{n},
  \end{align*}
  for large enough $n$, as desired.
\end{proof}

We are now left with proving Lemma \ref{Lem:1stDE-fourier}, which
relies on the following bound on the operator norm.

\begin{lemma} \label{Lem:gap}
There exists a positive constant $\g_G$ depending on $G$ such that for
  any $\r \in G^\ast$ and any $\s \in \Snon{1/6}$,
    \[ \|I_{d_\r}+X_\r(\s)\|_{op}  \le 1-\frac{\g_G}{n}. \]
\end{lemma}

\begin{proof}
  Let $\Delta_G$ denote the set of all probability distributions on
  $G$, and for $c \in (0, 1)$, let $\Delta_G(c) \subset \Delta_G$
  denote the set of all probability distributions $\mu$ such that
  $\mu(H) \le 1 - c$ for all proper subgroups $H \subset G$.

  Consider a representation $\r \in G^\ast$, and consider the function
  $h : \Delta_G(1/6) \to \End_\C(V_\r)$ given by
  \[ h(\mu) = \sum_{a \in G} \mu(a) \frac{\r(a)+\r(a)^\ast}{2}. \]
  Then, $h(\m)$ is hermitian, and since $\r$ is unitary, we clearly have
  \[\l(\m):=\max_{v \in V_\r, \|v\|=1}\langle h(\m)v, v\rangle \le 1. \]
  We claim that $\l(\m) < 1$ for each $\m \in \Delta_G(c)$.
  Indeed, suppose the contrary.
  Then, there exists a non-zero vector $v \in V_\r$ such that $\Re \langle \r(a)v, v \rangle=1$ for all $a \in G$ with $\m(a)>0$.
  This implies that the support of $\m$ is included in the subgroup
  \[H=\{a \in G \ : \ \r(a)v=v\}.\]
  Since $\r$ is a (non-trivial) irreducible representation, $H$ is a
  proper subgroup of $G$, and thus $\m(H) \le 1- c$, contradicting the
  assumption that $\m \in \Delta_G(c)$.

  Note that $\m \mapsto \l(\m)$ is continuous.
  We may define
  \[\g_\r:=\max_{\m \in \Delta_G(1/6)}\l(\m) < 1
  \qquad\text{and}\qquad \tilde \g_G:=\max_{\r \in G^\ast}\g_\r<1.\]
  Then, we have for any $\s \in \Snon{1/6}$,
  \[ \frac{x_\r(\s) + x_\r(\s)^\ast}{2} = \sum_{a \in G} \frac{n_a(\s)}{n}\frac{\r(a)+\r(a)^\ast}{2}
  \preceq \tilde \g_G I_{d_\r}. \]
  Taking $0<\g_G < 1-\tilde \g_G$, and
  plugging this into the definition of $X_\r$ gives
  $X_\r(\s) \preceq -\frac{\g_G}{n}I_{d_\r}$.
  Note that $X_\r(\s) \succeq -\frac{2}{n-1}I_{d_\r}$.
  Combining these together gives the result.
\end{proof}

\begin{remark}
  A much more direct approach is possible in the case $G = \Z/q$. The
  condition $\s \in \Snon{1/6}$ implies that $n_0(\s) \le
  \frac{5}{6}$. Then, we have
  \[ \Re x_\ell(\s) := \Re \sum_{a = 0}^{q - 1} \frac{n_a(\s)}{n} \o^{a \ell} \le \frac{5}{6} + \frac{1}{6} \max_{1 \le a \le q - 1} \Re \o^{a \ell} = \frac{5}{6} + \frac{1}{6} \cos \frac{2\pi}{q} < 1 - \g_G \]
  for some positive $\g_G$. Some rearranging of equations then yields
  the desired result.
\end{remark}

\begin{proof}[Proof of Lemma \ref{Lem:1stDE-fourier}]
  Fix $\r \in G^\ast$. Let $\Gc_t$ denote the event where for all $0
  \le s \le t$, we have $\|I_{d_\r}+X_\r(s)\|_{op} \le 1 - \frac{\g_G}{n}$,
  where $\g_G$ is taken as in Lemma \ref{Lem:gap}. Since our chain
  starts at $\s \in \Snon{1/3}$, Lemmas \ref{Lem:Burn} and
  \ref{Lem:gap} together imply that
  \[ \Pr_\s(\Gc_{n^2}^{\sf c}) \le C_G n^2 e^{-n/10}. \]
  Next, we turn to \eqref{Eq:DEfourier}. Rearranging
  \eqref{Eq:DEfourier} and squaring, we have
  \begin{align}
    \|x_\r(t+1)\|_{\HS}^2 &= \|x_\r(t)(I_{d_\r} + X_\r(t))\|_{\HS}^2 + \|\wh{M}_\r(t+1)\|_{\HS}^2 \nonumber \\
    &\hphantom{==} + 2\Re \langle x_\r(t)(I_{d_\r} + X_\r(t)), \wh{M}_\r(t+1) \rangle_{\HS} \label{Eq:1stDEsquared}
  \end{align}

  Let $z_t := \1_{\Gc_t} \|x_\r(t)\|_{\HS}^2$ and
  \[ M'(t+1) := \|\wh{M}_\r(t+1)\|_{\HS}^2 + 2\Re \langle x_\r(t)(I_{d_\r} + X_\r(t)), \wh{M}_\r(t+1) \rangle_{\HS}. \]
  Substituting into \eqref{Eq:1stDEsquared}, we obtain
  \[ z_{t+1} \le \|I_{d_\r} + X_\r(t)\|_{op}^2 \cdot z_t +  \1_{\Gc_t} M'(t+1) \le \left(1 - \frac{\g_G}{n}\right)^2 z_t +  \1_{\Gc_t} M'(t+1). \]
  Note that we have the bounds
  \[ \E[ M'(t+1) \mid \Fc_t ]= \E[ \|\wh{M}_\r(t+1)\|_{\HS}^2 \mid \Fc_t ] \le \frac{4\Qc^2d_\r}{n^2} \]
  \[ |M'(t+1)| \le \|\wh{M}_\r(t+1)\|_{\HS}^2 + 2\sqrt{d_\r}\left(1+\frac{1}{n(n-1)}\right) \|\wh{M}_\r(t+1)\|_{\HS} \le \frac{6\Qc^2 d_\r}{n}. \]

  We now apply Lemma \ref{Lem:GenDE} with $\e = \frac{1}{n}$, $\f(t) =
  \g_G$, $D = 6\Qc^2 d_\r$, and $\lambda = n^{1/4}$. This yields
  \[ \Pr\left( z_t \ge n^{-1/4} + e^{-\g_G t/n}\cdot z_0 \right) \le C'_G e^{-c'_G \sqrt{n}}. \]
  Consequently,
  \[ \Pr_\s\left( \|x_\r(t)\|_{\HS} \ge n^{-1/8} + e^{-\g_G t/2n} \cdot \|x_\r(0)\|_{\HS} \right) \le C'_G e^{-c'_G \sqrt{n}} + C_G n^2 e^{-n/10}. \]
  The lemma with $c_G = \g_G/2$ then follows from union bounding over
  all $1 \le t \le n^2$ and taking $n$ sufficiently large.
\end{proof}

\subsection{Proportion matrix chain: Proof of Lemma \ref{Lem:2ndDE}}\label{Subsec:matrix}

We carry out a similar albeit more refined strategy to analyze the
proportion matrix. Throughout this section, we assume our Markov chain
$(\s_t)_{t \ge 0}$ starts at an initial state $\s_\ast \in
\Sast{\frac{1}{4\Qc}}$. We again write $n_a(t)=n_a(\s_t)$ and
$n_{a, b}(t)=n^{\s_\ast}_{a, b}(\s_t)$, and similar to before, the
$n_{a, b}(t)$ satisfy the difference equation
\begin{equation}\label{Eq:de2}
n_{a, b}(t+1)-n_{a, b}(t)=\sum_{c \in G}\frac{n_{a, bc^{-1}}(t)n_c(t)}{2n(n-1)}+\sum_{c \in G}\frac{n_{a, bc}(t)n_c(t)}{2n(n-1)} - \frac{n_{a, b}(t)}{n} + M_{a, b}(t+1),
\end{equation}
where $\E[M_{a, b}(t+1) \mid \Fc_t]=0$ and $|M_{a, b}(t)| \le 2$ for all $t \ge 0$.

We can again analyze this equation via the Fourier transform. In this
case, for each $a \in G$, we take the Fourier transform of
$\left(n_{a, b}(t)/n_a(\s_\ast)\right)_{b \in G}$. For $\r \in G^\ast$,
let
\[ y_{a,\r}(t) = y_{a,\r}^{\s_\ast}(t) := \sum_{b \in G}\frac{n_{a, b}(t)}{n_a(\s_\ast)}\r(b) \]
denote the Fourier coefficient at $\r$. Let $\wh{M}_{a, \r}(t) :=
\frac{1}{n_a(\s_\ast)}\sum_{b \in G}M_{a, b}(t)\r(b)$. Then, \eqref{Eq:de2} becomes
\begin{equation}\label{Eq:DEy}
y_{a, \r}(t+1) - y_{a, \r}(t) = y_{a, \r}(t) X_\r(t) + \wh{M}_{a, \r}(t+1).
\end{equation}
Note that $\E_\s[\wh{M}_{a, \r}(t+1) \mid \Fc_t]=0$. Also, since we
assumed $\s_\ast \in \Sast{\frac{1}{4\Qc}}$, it follows that $\frac{n_a(\s_\ast)}{n}
\ge \frac{1}{2\Qc}$. Thus, we also know
$\|\wh{M}_{a, \r}(t+1)\|_{\HS} \le \frac{4\Qc^2\sqrt{d_\r}}{n}$.

Again, our main step is a bound on the Fourier coefficients
$y_{a, \r}(t)$, which will also be useful later in proving Lemma
\ref{Lem:RW}.

\begin{lemma} \label{Lem:2ndDE-fourier}
  Consider any $\s_\ast, \s'_\ast \in \Sast{\frac{1}{4\Qc}}$. There
  exist constants $c_G, C_G > 0$ depending only on $G$ such that for
  all large enough $n$, we have
  \[ \Pr_{\s'_\ast}\left( \|y^{\s_\ast}_{a, \r}(t)\|_{\HS} \ge R \left(\frac{1}{\sqrt{n}} + e^{-t/n} \|y^{\s_\ast}_{a, \r}(0)\|_{\HS} \right) \right) \le e^{-\O_G(R^2) + O_G(1)} + \frac{2}{n^3} \]
  for all $t$ and $R > 0$.
\end{lemma}

The above lemma directly implies Lemma \ref{Lem:2ndDE}.

\begin{proof}[Proof of Lemma \ref{Lem:2ndDE}]
  We apply Lemma \ref{Lem:2ndDE-fourier} to each $a \in G$ and $\r \in
  G^\ast$. Recall that $T = \ceil{\frac{1}{2} n \log n}$, so that
  \[ \frac{1}{\sqrt{n}} + e^{-T/n} \|y^{\s_\ast}_{a, \r}(0)\|_{\HS} \le \frac{2\sqrt{d_\r}}{\sqrt{n}}. \]
  Then, Lemma \ref{Lem:2ndDE-fourier} implies
  \[ \Pr_{\s'_\ast}\left( \|y^{\s_\ast}_{a, \r}(T)\|_{\HS} \ge \frac{R}{\sqrt{n}} \right) \le e^{-\O_G(R^2) + O_G(1)} + \frac{2}{n^3}. \]
  Union bounding over all $a \in G$ and $\r \in G^\ast$ and using the
  Plancherel formula, this yields
  \begin{align*}
    &\Pr_{\s'_\ast}\left( \s_\ast \not\in \Sast{\s_\ast, \frac{R}{\sqrt{n}}} \right) \le \Pr_{\s'_\ast}\left( \max_{a, \r} \|y^{\s_\ast}_{a, \r}(T)\|_{\HS} \ge \frac{R}{\sqrt{n}} \right) \\
    &\qquad\qquad \le e^{-\O_G(R^2) + O_G(1)} + \frac{2 \Qc^2}{n^3} \le C_G e^{-R} + \frac{1}{n}
  \end{align*}
  for sufficiently large $C_G$ and $n$.
\end{proof}

We now prove Lemma \ref{Lem:2ndDE-fourier}. Before proceeding with the
main proof, we need the following routine estimate as a preliminary
lemma.

\begin{lemma} \label{Lem:theta}
  Let $\theta_n : \R^d \to \R^+$ be the function given by $\theta_n(x)
  = \|x\| + \frac{1}{\sqrt{n}}e^{-\sqrt{n}\|x\|} -
  \frac{1}{\sqrt{n}}$. Then, we have the inequalities
  \[ \|\nabla \theta_n(x)\| \le 1, \qquad \theta_n(x + h) \le \theta_n(x) + \langle h, \nabla \theta_n(x) \rangle + \frac{\sqrt{n}}{2} \|h\|^2. \]
\end{lemma}
\begin{proof}
  We can write $\theta_n(x) = f(\|x\|)$, where $f(r) = r +
  \frac{1}{\sqrt{n}} e^{-\sqrt{n} r} - \frac{1}{\sqrt{n}}$. By
  spherical symmetry, we have
  \[ \|\nabla \theta_n(x)\| = f'(\|x\|) = 1 - e^{-\sqrt{n}\|x\|} \le 1, \]
  which is the first inequality. Again by spherical symmetry, the
  eigenvalues of the Hessian $\nabla^2 \theta_n(x)$ can be directly
  computed to be $f''(\|x\|)$ and $f'(\|x\|) / \|x\|$. But these are
  bounded by
  \[ f''(r) \le \sqrt{n} e^{-\sqrt{n}r} \le \sqrt{n}, \qquad f'(r)/r \le \frac{1 - e^{-\sqrt{n}r}}{r} \le \sqrt{n}. \]
  Thus, $\nabla^2 \theta_n(x) \preceq \sqrt{n} I$, and the second inequality
  follows from Taylor expansion.
\end{proof}

\begin{proof}[Proof of Lemma \ref{Lem:2ndDE-fourier}]
  Let $\g_G$ and $c_G$ be the constants from Lemmas \ref{Lem:gap} and
  \ref{Lem:1stDE-fourier}, respectively. Define the events
  \[ \Gc_t := \bigcap_{s = 0}^t \left\{ X_\r(\s_s) \preceq -\frac{\g_G}{n} \right\}, \qquad \Gc'_t := \bigcap_{s = 0}^t \left\{ X_\r(\s_s) \preceq -\frac{1 - \sqrt{d_\r} e^{-c_G s/n} - 2n^{-1/8}}{n} \right\}. \]

  Note that $\s'_\ast \in \Sast{\frac{1}{4\Qc}} \subseteq
  \Snon{1/3}$. Hence, by Lemmas \ref{Lem:Burn} and \ref{Lem:gap}, we
  have $\Pr(\Gc^{\sf c}_{n^2}) \le C_G n^2 e^{-n/10}$. We also have
  \begin{align*}
    X_\r(s) &= \frac{1}{n - 1} \left( \frac{x_\r(s)+x_\r(s)^\ast}{2} - \frac{n - 1}{n}I_{d_\r} \right) \preceq -\frac{1}{n}\left( 1 - \frac{n \|x_\r(s)\|_{\HS}}{n - 1} \right) I_{d_\r} \\
    &\preceq -\frac{1}{n}\left( 1 - \|x_\r(s)\|_{\HS} - \frac{\sqrt{d_\r}}{n-1} \right) I_{d_\r},
  \end{align*}
  where we have used the fact that
  $\left\|\frac{x_\r(s)+x_\r(s)^\ast}{2}\right\|_{op} \le \|x_\r(s)\|_{op} \le
  \|x_\r(s)\|_{\HS}$.

  Lemma \ref{Lem:1stDE-fourier} then implies that $\Pr(\Gc'^{\sf
    c}_{n^2}) \le \frac{1}{n^3}$. Thus, setting
  \[ \f(t) := \max(\g_G, 1 - \sqrt{d_\r} e^{-c_G t/n} - 2 n^{-1/8}), \]
  \[ \Hc_t := \Gc_t \cap \Gc'_t = \bigcap_{s = 0}^t \left\{ X_\r(\s_s) \preceq -\frac{\f(t)}{n} \right\}, \]
  we conclude that
  \[ \Pr(\Hc^{\sf c}_{n^2}) \le \Pr(\Gc^{\sf c}_{n^2}) + \Pr(\Gc'^{\sf c}_{n^2}) \le \frac{2}{n^3} \]
  for all large enough $n$.

  Next, we turn to \eqref{Eq:DEy} and apply $\theta_n$ to both sides,
  where we identify $\C^{d_\r^2}$ with $\R^{2d_\r^2}$. Using Lemma \ref{Lem:theta} and
  taking the conditional expectation, we obtain
  \begin{align*}
    \E\left[ \theta_n\left( y_{a, \r}(t+1) \right) \,\middle|\, \Fc_t \right] &\le \theta_n\left( y_{a, \r}(t) (I_{d_\r} + X_\r(t))\right) + \frac{8 \Qc^4 d_\r}{n \sqrt{n}} \\
    &\le \theta_n(\|I_{d_\r} + X_\r(t)\|_{op} \cdot y_{a, \r}(t)) + \frac{8 \Qc^4 d_\r}{n \sqrt{n}} \\
    &\le \|I_{d_\r} + X_\r(t)\|_{op} \cdot \theta_n(y_{a, \r}(t)) + \frac{8 \Qc^4 d_\r}{n \sqrt{n}},
  \end{align*}
  where the second inequality follows from the variational formula for
  operator norm (i.e. that $\|BA\|_{\HS} \le \|A\|_{op} \|B\|_{\HS}$),
  and the third inequality follows from the fact that $\theta_n$ is
  convex with $\theta_n(0) = 0$.  Thus, we may write
  \[ \theta_n(y_{a, \r}(t+1)) \le \|I_{d_\r} + X_\r(t)\|_{op} \cdot \theta_n(y_{a, \r}(t)) + M'(t+1) \]
  where
  \[ \E[ M'(t+1) \mid \Fc_t ] \le \frac{8 \Qc^4 d_\r}{n\sqrt{n}}, \qquad |M'(t+1)| \le \frac{8\Qc^2 \sqrt{d_\r}}{n}. \]

  Now, let $z_t := \1_{\Hc_t} \theta_n(y_{a, \r}(t))$, and note that
  since $X_\r(\s) \succeq -\frac{2}{n-1} I_{d_\r}$, we have
  $\|I_{d_\r} + X_\r(t)\|_{op} \le 1-\frac{\f(t)}{n}$ whenever $\Hc_t$
  holds. Thus,
  \[ z_{t+1} \le \|I_{d_\r} + X_\r(t)\|_{op} \cdot z_t + \1_{\Hc_t}M'(t+1) \le \left(1 - \frac{1}{n} \f(t) \right) z_t + \1_{\Hc_t}M'(t+1). \]

  We may then apply Lemma \ref{Lem:GenDE} with $\e = \frac{1}{n}$ and
  $D = 8\Qc^4 d_\r$. Note that
  \begin{align*}
    \int_0^t \f(s) \,ds &\ge \left(1 - 2 n^{-\frac{1}{8}}\right)t - \sqrt{d_\r} \int_0^\infty e^{-\frac{c_G s}{n}} \,ds \ge t - O_G(n)
  \end{align*}
  for all large enough $n$. Thus, Lemma \ref{Lem:GenDE} implies that
  \begin{equation} \label{Eq:2ndDE-fourier-z_t}
    \Pr\left(z_t \ge \frac{\lambda}{\sqrt{n}} + C_G e^{-t/n} \cdot z_0 \right) \le C'_G e^{-c'_G \lambda^2}.
  \end{equation}
  Consequently,
  \begin{align*}
    &\Pr\left( \|y_{a, \r}(t)\|_{\HS} \ge R\left( \frac{1}{\sqrt{n}} + e^{-\frac{t}{n}}\|y_{a, \r}(0)\|_{\HS} \right) \right) \\
    &\qquad\qquad \le \Pr\left( \theta_n(y_{a, \r}(t)) \ge \frac{R - 1}{\sqrt{n}} + Re^{-\frac{t}{n}}\|y_{a, \r}(0)\|_{\HS} \right) \\
    &\qquad\qquad \le \Pr\left( \theta_n(y_{a, \r}(t)) \ge \frac{R - 1}{\sqrt{n}} + Re^{-\frac{t}{n}}\theta_n(y_{a, \r}(0)) \right) \\
    &\qquad\qquad \le \Pr\left( z_t \ge \frac{R - 1}{\sqrt{n}} + Re^{-\frac{t}{n}}z_0 \right) + \Pr(\Hc^{\sf c}_{n^2}) \\
    &\qquad\qquad \le e^{-\O_G(R^2) + O_G(1)} + \frac{2}{n^3},
  \end{align*}
  as desired.
\end{proof}

\section{Construction of the coupling: Proof of Lemma \ref{Lem:RW}}\label{Sec:Coupling-Proof}

For each $\d>0$, we define a subset of $\{0, 1, \dots, n\}^{G \times
  G}$ by
\[ \Mc_\d:=\left\{ (n_{a, b})_{a, b \in G} \ : \ n_{a, b} \ge \frac{(1-\d) n}{\Qc^2} \ \text{for every $a, b \in G$ and}\ \sum_{a, b \in G}n_{a, b}=n\right\}. \]

\begin{lemma}\label{Lem:coupling}
  Consider a configuration $\s_\ast \in \Sc$ and a constant $0<\d \le
  \frac{1}{2\Qc^2}$, and assume that $(1 - \d)n/\Qc^2$ is an
  integer. Let $(\s_t)_{t \ge 0}$ and $(\tilde\s_t)_{t \ge 0}$ be two
  product replacement chains started at $\s$ and $\tilde\s$,
  respectively. Then, there exists a coupling $(\s_t, \tilde \s_t)$ of
  the Markov chains satisfying the following:

  Let
  \[ D_t:=\frac{1}{2}\sum_{a, b \in G}|n^{\s_\ast}_{a, b}(\s_t) - n^{\s_\ast}_{a, b}(\tilde \s_t)|. \]
  Then, on the event $\{(n^{\s_\ast}_{a, b}(\s_t))_{a, b \in G}, (n^{\s_\ast}_{a, b}(\tilde \s_t))_{a, b \in G} \in \Mc_\d\}$ and $\{D_t > 0\}$, one has
  \begin{align}
    \E_{\s, \tilde \s}[D_{t+1}-D_t \mid \s_t, \tilde \s_t] &\le 0, \label{Eq:D-drift-0} \\
    \Pr_{\s, \tilde \s}\left(D_{t+1} - D_t \neq 0 \mid \s_t, \tilde \s_t \right) &\ge \frac{(1-\d)^2}{4\Qc^3}. \label{Eq:D-fluctuate}
  \end{align}
\end{lemma}

\begin{proof}
  Let us abbreviate $n_{a, b}(t) = n^{\s_\ast}_{a, b}(\s_t)$ and $\tilde
  n_{a, b}(t) = n^{\s_\ast}_{a, b}(\tilde \s_t)$. Let
  $m_{a, b}(t):=\min(n_{a, b}(t), \tilde n_{a, b}(t))$. For each $a \in G$, we
  define the quantity
  \[ d_a(t) := \frac{1}{2}\sum_{b \in G} |n_{a, b}(t) - \tilde{n}_{a, b}(t)| = \sum_{b \in G} n_{a, b}(t) - \sum_{b \in G} m_{a, b}(t), \]
  so that $D_t = \sum_{a \in G} d_a(t)$.

  For accounting purposes, it is helpful to introduce two sequences
  \[ (x_1, x_2, \ldots , x_n) \text{ and } (\tilde{x}_1, \tilde{x}_2, \ldots , \tilde{x}_n) \]
  of elements of $G \times G$. These sequences are chosen so that
  the number of $x_k$ equal to $(a, b)$ is exactly $n_{a, b}$, and
  similarly the number of $\tilde{x}_k$ equal to $(a, b)$ is
  $\tilde{n}_{a, b}$.  Moreover, we arrange their indices in a coordinated
  fashion, as described below.

  We define three families of disjoint sets: $P_{a, b}$, $Q_a$, and $R_a
  \subset [n]$.
  \begin{itemize}
  \item For each $a, b \in G$, let $P_{a, b}$ be a set of size $(1 - \delta)n/\Qc^2$ such that for any $k \in P_{a, b}$, we have $x_k
    =\tilde{x}_k = (a, b)$. (This is possible provided that $(n_{a, b}(t)), (\tilde n_{a, b}(t)) \in \Mc_\d$
    holds.)
  \item For each $a \in G$, let $Q_a$ be a set of size $\sum_{b \in G}(m_{a, b} - |P_{a, b}|)$ such that for any $k \in Q_a$, $x_k
    =\tilde{x}_k= (a, b)$ for some $b$. (Note that $Q_a$ may be empty.)
  \item For each $a \in G$, let $R_a$ be a set of size $d_a$ such
    that for any $k \in R_a$, $x_k$ and $\tilde{x}_k$ both have $a$ as
    their first coordinate. (This $R_a$ is well-defined since $\sum_b
    n_{a, b} = \sum_b \tilde n_{a, b}$ for each $a$; it may also be empty.)
  \end{itemize}
  Define
  \[
  P := \bigsqcup_{a, b \in G} P_{a, b}, \qquad Q := \bigsqcup_{a \in G} Q_a, \qquad R := \bigsqcup_{a \in G} R_a.
  \]
  Suppose that $D_t > 0$, so that for some $a_*, b_*, b_*' \in G$ we
  have $n_{a_*, b_*} > \tilde{n}_{a_*, b_*}$ and $n_{a_*, b'_*} <
  \tilde{n}_{a_*, b'_*}$. Let us consider all possible ways to sample a
  pair of indices and a sign $(k, l, s) \in \{1, 2, \ldots , n\}^2
  \times \{\pm 1\}$ with $k \neq l$.

  Suppose $x_k = (a_k, b_k)$ and $x_l = (a_l, b_l)$. We think of $(k, l,
  +1)$ as corresponding to a move on $(n_{a, b}(t))$ where $n_{a_k, b_k}$ is
  decremented and $n_{a_k, (b_k \cdot b_l)}$ is incremented. Similarly, $(k, l,
  -1)$ corresponds to a move where $n_{a_k, b_k}$ is decremented and
  $n_{a_k, (b_k \cdot b_l^{-1})}$ is incremented. We may also think of $(k, l, \pm
  1)$ as corresponding to moves on $(\tilde n_{a, b}(t))$ in an analogous way.

  \begin{figure}
    \includegraphics{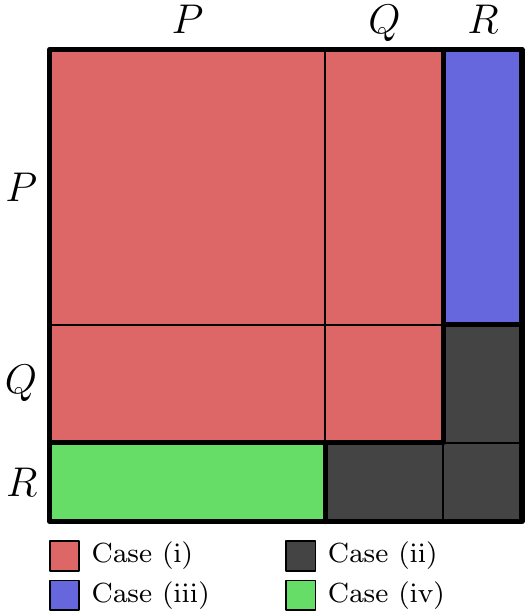}
    \caption{Illustration of cases (i) through (iv).}
    \label{fig:coupling-cases}
  \end{figure}

  We now analyze four cases, as illustrated in Figure \ref{fig:coupling-cases}.

  \paragraph{(i) {\bf Case $(k, l) \in (P \sqcup Q) \times (P \sqcup Q)$.}}

  For all but an exceptional situation described below, we apply the
  move corresponding to $(k, l, s)$ to both states $(n_{a, b}(t))$ and
  $(\tilde n_{a, b}(t))$. In these cases, $D_{t+1}=D_t$.

  We now describe the exceptional situation. Define
  \[ S = P_{a_*, b_*} \times \left(\bigsqcup_{c \in G} P_{c, (b_*^{-1} \cdot b'_*)}\right) \qquad\text{and}\qquad S' = P_{a_*, b'_*} \times \left(\bigsqcup_{c \in G} P_{c, \id}\right). \]
  Then, the exceptional situation occurs when $s = +1$ and $(k, l) \in
  S \sqcup S'$.

  Take any bijection $\tau$ from $S$ to $S'$. If $(k, l) \in S$, then
  we apply $(k, l, +1)$ to $(n_{a, b}(t))$ while applying $(\tau(k, l),
  +1)$ to $(\tilde n_{a, b}(t))$. This increments $n_{a_*, b'_*}$,
  decrements $n_{a_*, b_*}$, and has no effect on the
  $(\tilde{n}_{a, b}(t))$. The overall effect is that $D_{t+1} = D_t - 1$.

  If instead $(k, l) \in S'$, then we apply $(k, l, +1)$ to $(n_{a,
    b}(t))$ and $(\tau^{-1}(k, l), +1)$ to $(\tilde n_{a, b}(t))$. A
  similar analysis shows that in this case $D_{t+1} = D_t + 1$.

  The exceptional event occurs with probability $\frac{(1 -
    \d)^2}{2\Qc^3}$, and when it occurs, $D_t$ increases or decreases
  by $1$ with equal probability. Thus, the exceptional situation plays
  the role of introducing some unbiased fluctuation in $D_t$ and gives
  us \eqref{Eq:D-fluctuate}. \\

  \paragraph{(ii) {\bf Case $(k, l) \in (Q \sqcup R) \times (Q \sqcup R)$ but $(k, l) \not\in Q \times Q$.}}

  This occurs with probability
  \[ \frac{1}{n(n-1)}\left((|Q|+|R|)(|Q|+|R|-1)-|Q|(|Q|-1)\right) \]
  which is at most
  \[ \frac{2}{n(n - 1)}(|Q| + |R|) |R|  = \frac{2 \d}{n - 1}D_t. \]
  Apply the move corresponding to $(k, l, s)$ to both states.  This
  increases $D_t$ by at most $1$. We will see later that the effect of
  this case is small compared to the other cases. \\

  \paragraph{(iii) {\bf Case $(k, l) \in P \times R$.}}

  This occurs with probability
  \[ \frac{1}{n(n - 1)}|P| |R| =\frac{1 - \d}{n - 1}D_t. \]
  Apply the move corresponding to $(k, l, s)$ to both states. Again,
  this increases $D_t$ by at most $1$, but there is also a chance not
  to increase.

  Suppose that $x_l = (a_1, b_1)$ and $\tilde{x}_l = (a_1,
  \tilde{b}_1)$, and suppose that $k \in P_{a_2, b_2}$. Then the move has
  the effect of decreasing $n_{a_2, b_2}$ and $\tilde{n}_{a_2, b_2}$ while
  increasing $n_{a_2, (b_2 \cdot b_1^s)}$ and $\tilde{n}_{a_2,
    (b_2\cdot \tilde{b}_1^s)}$.
  Note that conditioned on this case happening, $(a_2, b_2)$ is
  distributed uniformly over $G \times G$. When $(a_2,
  (b_2\cdot \tilde{b}_1^s)) = (a_*, b_*)$ or $(a_2,
  (b_2\cdot b_1^s)) = (a_*, b'_*)$, the move does not increase
  $D_t$. Therefore there is at least a $2/\Qc^2$ chance that $D_t$ is
  actually not increased. Hence, the probability that $D_t$ is
  increased by $1$ is at most
  \[ \left(1 - \frac{2}{\Qc^2}\right)\frac{1 - \d}{n-1}D_t. \]

  \paragraph{(iv) {\bf Case $(k, l) \in R \times P$.}}

  This occurs with probability
  \[ \frac{1}{n(n - 1)} |R| |P| = \frac{1 - \d}{n - 1} D_t. \]
  Suppose that $x_k = (a, b)$ and $\tilde{x}_k = (a, \tilde{b})$.  Let
  $\tau$ be a permutation of $P$ such that for $l \in P_{a, c}$, one has
  $\tau(l) \in P_{a, \tilde{b}^{-1}\cdot b \cdot c^s}$.  Then apply $(k, l, s)$ to $(n_{a, b}(t))$
  and apply $(k, \tau(l), s)$ to $(\tilde n_{a, b}(t))$.  This always decreases
  $D_t$ by $1$. \\

  Let us now summarize what we know when $(n_{a, b}(t)), (\tilde n_{a, b}(t)) \in \Mc_\d$ and $D_t>0$.
  From Cases (i), (ii), and (iii), we have
  \[ \Pr_{\s, \tilde \s}(D_{t+1} = D_t + 1 \mid \s_t, \tilde \s_t) \le \left(1-\frac{2(1-\d)}{\Qc^2}+\d\right)\frac{D_t}{n-1} + \frac{(1 - \d)^2}{4\Qc^3}. \]
  From Cases (i) and (iv), we have
  \[ \Pr_{\s, \tilde \s}(D_{t+1} = D_t - 1 \mid \s_t, \tilde \s_t) \ge (1-\d)\frac{D_t}{n-1}  + \frac{(1 - \d)^2}{4\Qc^3}. \]
  Therefore, if $0<\d \le \frac{1}{2\Qc^2}$, then
  \[\E_{\s, \tilde \s}[D_{t+1}-D_t \mid \s_t, \tilde \s_t] \le 0, \]
  verifying \eqref{Eq:D-drift-0}.

  To fully define the coupling, when $D_t = 0$, we can couple $\s_t$
  and $\s_t$ to be identical, and if either $(n_{a, b}(t)) \notin
  \Mc_\d$ or $(\tilde n_{a, b}(t)) \notin \Mc_\d$, we may run the two
  chains independently.
\end{proof}

\begin{proof}[Proof of Lemma \ref{Lem:RW}]
  Since $\s \in \Sast{\s_\ast, \frac{R}{\sqrt{n}}}$, we must have for
  each $a \in G$ and $\r \in G^\ast$ that $\|y^{\s_\ast}_{a, \r}(\s)\|_{\HS}
  \le \frac{R}{\sqrt{n}}$. Note that for large enough $n$, we have
  $\Sast{\s_\ast, \frac{R}{\sqrt{n}}} \subseteq
  \Sast{\frac{1}{5\Qc^3}}$. Thus, we may apply Lemma
  \ref{Lem:2ndDE-fourier} to obtain
  \begin{equation} \label{Eq:RW-1}
    \Pr\left( \bigcup_{t = 0}^{n^2} \left\{ \|y^{\s_\ast}_{a, \r}(\s_t)\|_{\HS} \ge \frac{1}{5\Qc^3} \right\} \right) \le n^2 \left( e^{-\O_{G}(n) + O_{G}(1)} + \frac{2}{n^3} \right) \le \frac{3}{n}
  \end{equation}
  for large enough $n$. Define the event
  \[ \Gc_t := \left\{ \s_s \in \Sast{\s_\ast, \frac{1}{5\Qc^3}} \text{ for all $1 \le s \le t$} \right\}. \]
  The Plancherel formula applied to \eqref{Eq:RW-1} implies that
  $\Pr(\Gc^{\sf c}_{n^2}) \le \frac{3\Qc^2}{n}$. We may analogously define
  an event $\tilde\Gc_t$ for $\tilde\s$ and let $\Ac_t := \Gc_t \cap
  \tilde\Gc_t$. Thus, $\Pr(\Ac_{n^2}^{\sf c}) \le \frac{6\Qc^2}{n}$.

  Pick $\d' \in \left(\frac{2}{5\Qc^2}, \frac{3}{7\Qc^2}\right)$ so that $(1 - \d')n/\Qc^2$ is an
  integer. Note that when $\Ac_t$ holds, we have
  \[ \s_t \in \Sast{\s_\ast, \frac{1}{5\Qc^3}} \quad \text{and} \quad \s_\ast \in \Sast{\frac{1}{5\Qc^3}}\implies (n_{a, b}(t)) \in \Mc_{\frac{2}{5\Qc^2}} \subseteq \Mc_{\d'}, \]
  and similarly $\tilde\s_t \in \Mc_{\d'}$.

  Thus, we may invoke Lemma \ref{Lem:coupling} to give a coupling
  between $\s$ and $\tilde\s$ where on the event $\Ac_t$, the quantity
  $D_t$ is more likely to decrease than increase. Letting ${\bf
    D}_t:={\bf 1}_{\Ac_t}D_t$, we see that $({\bf D}_t)$ is a
  supermartingale with respect to $(\Fc_t)$.

  Define
  \[ \t := \min\{ t \ge 0 : D_t=0 \}, \qquad {\tilde \t} := \min\{ t \ge 0 : {\bf D}_t=0 \}. \]
  Then, Lemma \ref{Lem:coupling} ensures that on the event $\{\tilde\t
  > t\}$, we have $\Var({\bf D}_{t+1}\mid \Fc_t) \ge \a^2$, where
  $\a^2:=\left(1- \frac{1}{\Qc^2}\right)\frac{(1-\d')^2}{4\Qc^3}$. By
  \cite[Proposition 17.20]{LevinPeresWilmer}, for every $u > 12/\a^2$,
  \begin{equation}\label{Eq:RW}
    \Pr(\tilde \t > u) \le \frac{4 {\bf D}_0}{\a\sqrt{u}}.
  \end{equation}
  Recall that $T = \ceil{\b n}$ and $D_0 \le \sqrt{\Qc} R\sqrt{n}$. As
  long as $\b$ is large enough, we may apply \eqref{Eq:RW} with $u =
  T$ to get
  \[ \Pr_{\s, \tilde \s}(\t > T) \le \frac{16 \Qc^2 R}{(1-\d')\sqrt{\b}} + \Pr(\Ac_T^{\sf c}) \le \frac{32 \Qc^2 R}{\sqrt{\b}} \]
  for all large enough $n$, as desired.
\end{proof}

\section{Proof of Theorem \ref{Thm:main} (\ref{Eq:LB})}\label{Sec:LB}

\newcommand{\nnon}{n^{\{\id\}}_{non}}

The lower bound is proved essentially by showing that the estimates of
Lemmas \ref{Lem:Burn} and \ref{Lem:2ndDE} cannot be improved. Let
$a_1, a_2, \ldots , a_k$ be a set of generators for $G$. Let $\s_\star
\in \Sc$ be the configuration given by
\[ \s_\star(i) = \begin{cases}
  a_i & \text{if $i \le k$,} \\
  0 & \text{otherwise}.
\end{cases} \]
We will analyze the Markov chain started at $\s_\star$ and show that
it does not mix too fast.

Recall from Section \ref{Sec:UB} the notation
\[ \nnon(\s) = |\{ i \in [n] : \s(i) \ne \id \}| \]
for the number of sites in $\s$ that do not contain the identity. We
first show that if we run the chain for slightly less than $n \log n$
steps, most of the sites will still contain the identity.

\begin{lemma} \label{Lem:Burn-lower}
  Let $T := \floor{n \log n - Rn}$. Then,
  \[ \Pr_{\s_\star}\left( \nnon(\s_T) \ge \frac{n}{3} \right) \le \frac{4\Qc^2}{R^2}. \]
\end{lemma}
\begin{proof}
  Recall that in one step of our Markov chain, we pick two indices $i,
  j \in [n]$ and replace $\s(i)$ with $\s(i) \cdot \s(j)$ or $\s(i) \cdot \s(j)^{-1}$.
  The only way for $\nnon(\s_t)$ to increase after this step is
  if $\s(j) \ne \id$. Thus,
  \begin{equation} \label{Eq:nnon}
    \Pr( \nnon(\s_{t+1}) = \nnon(\s_t) + 1 \mid \nnon(\s_t)) \le \frac{\nnon(\s_t)}{n}.
  \end{equation}

  Let $\t := \min\{ t \ge 0 : \nnon(\s_t) \ge \frac{n}{3} \}$ be the
  first time that $\nnon(\s_t)$ is at least $\frac{n}{3}$. We have
  that $\nnon(\s_\star) = k$, so it follows from \eqref{Eq:nnon} that
  $\t$ stochastically dominates the sum
  \[ G := \sum_{s = k}^{\floor{n/3}} G_s, \]
  where the $G_s$ are independent geometric variables with success
  probability $\frac{s}{n}$. Note that we have the bounds
  \[ \E G = \sum_{s = k}^{\floor{n/3}} \frac{n}{s} \ge n \left( \log \floor{\frac{n}{3}} - \log k \right), \qquad \Var(G) = \sum_{s = k}^{\floor{n/3}} \frac{n(n-s)}{s^2} \le n^2. \]
  Hence,
  \begin{align*}
    \Pr(\t < T) &\;\le\; \Pr(G < T) \;\le\; \Pr(G < \E G + n\log(3k) - Rn ) \\
    &\;\le\; \frac{n^2}{n^2(R - \log(3k))^2} \le \frac{4}{R^2}
  \end{align*}
  for $R \ge 2 \Qc \ge 2 \log (3k)$. On the other hand, the bound
  claimed in the lemma is trivial for $R \le 2 \Qc$, so we have
  completed the proof.
\end{proof}

Next, we show that it really takes about $\frac{1}{2}n \log n$ steps
for the Fourier coefficients $x_\r$ to decay to $O\left(
\frac{1}{\sqrt{n}} \right)$, as suggested by Lemma
\ref{Lem:2ndDE}. Note that it suffices here to analyze the $x_\r$
instead of the $y_{a, \r}$, which simplifies our analysis.  Actually,
it suffices to consider (the real part of) the trace of $x_\r$.  Here
the orthogonality of characters reads $\frac{1}{\Qc}\sum_{a \in G} \Tr
\r(a)=0$, and it takes about $\frac{1}{2}n \log n$ steps for $\Re \Tr
x_\r(t)$ to decay to $O\left( \frac{1}{\sqrt{n}} \right)$.

\begin{lemma} \label{Lem:DE-lower}
  Consider any $\r \in G^\ast$ and any $R > 5$. Let $T :=
  \floor{\frac{1}{2}n \log n - Rn}$, and suppose that $\s \in \Sc$
  satisfies $\nnon(\s) \le \frac{n}{3}$. Then,
  \[ \Pr_{\s}\left( \|x_\r(\s_T)\|_{\HS} \le \frac{R}{\sqrt{n}} \right) \le \frac{4\Qc^2}{R^2}. \]
\end{lemma}
\begin{proof}
  Let $z(t) := (1/d_\r) \Tr (x_\r(t)+x_\r(t)^\ast)/2$.
  Then, noting that \eqref{Eq:DEfourier} also holds for $x_\r(t)^\ast$ since $x_{\r^\ast}(t)=x_\r(t)^\ast$,
  we have
  \[ z(t+1)-z(t) = \frac{1}{n-1}\frac{1}{d_\r} \Tr \left(\frac{x_\r(t)+x_\r(t)^\ast}{2}\right)^2-\frac{1}{n}z(t) + M(t+1), \]
  where
  \[ \E[M(t+1) \mid \Fc_t]=0 \qquad \text{and} \qquad |M(t)| \le \frac{2\Qc}{n}. \]
  Here we have
  \[\frac{1}{d_\r} \Tr \left(\frac{x_\r(t)+x_\r(t)^\ast}{2}\right)^2 \ge z(t)^2.\]
  We compare $z(t)$ to another process $(w(t))_{t \ge 0}$ defined by
  $w(0) := \frac{1}{3}$ and
  \begin{equation} \label{Eq:w-equation}
    w(t+1) := \left(1-\frac{1}{n}\right)w(t) + M(t+1).
  \end{equation}

  We will show by induction that $z(t) \ge w(t)$ for all $t$. For the
  base case, note that since $\nnon(\s) \le \frac{n}{3}$, we have
  \[ z(0) = \frac{1}{d_\r}\Re \Tr \sum_{a \in G} \frac{n_a(t)}{n} \cdot \r(a) \ge \frac{2}{3} - \frac{1}{3} = \frac{1}{3}. \]
  Suppose now that $z(t) \ge w(t)$. Then,
  \begin{align*}
    z(t+1) &\ge z(t) + \frac{1}{n-1}z(t)^2-\frac{1}{n}z(t) + M(t+1) \\
      		&\ge \left(1 - \frac{1}{n}\right) w(t) + M(t+1) = w(t+1),
  \end{align*}
  completing the induction.

  It now suffices to lower bound $w(T)$. To this end, we first note
  that applying \eqref{Eq:w-equation} repeatedly and taking
  expectations, we obtain
  \[ \E w(T) = \left(1 - \frac{1}{n}\right)^T \cdot \frac{1}{3} \ge \frac{e^R}{6\sqrt{n}} \ge \frac{2R}{\sqrt{n}}. \]
  In order to calculate the variance of $w(T)$, we can also square
  \eqref{Eq:w-equation} and take the expectation, which gives us
  \begin{align*}
    \Var(w(T)) &= \E w(T)^2 - (\E w(T))^2 \\
    &\le \left(1 - \frac{1}{n}\right)^{2T} \cdot \frac{1}{9} + n \cdot \left( \frac{2\Qc}{n} \right)^2  -  \left( \left(1 - \frac{1}{n}\right)^T
  \cdot \frac{1}{3} \right)^2 \\
    &= \frac{4\Qc^2}{n}.
  \end{align*}
  Then, by Chebyshev's inequality, we have
  \begin{align*}
    \Pr_{\s}\left( \|x_\r(\s_T)\|_{\HS} \le \frac{R}{\sqrt{n}} \right) &\le \Pr\left( z(T) \le \frac{R}{\sqrt{n}} \right) \le \Pr\left( w(T) \le \frac{R}{\sqrt{n}} \right) \\
    &\le \frac{4\Qc^2/n}{(R/\sqrt{n})^2} = \frac{4\Qc^2}{R^2},
  \end{align*}
  as desired.
\end{proof}

\begin{proof}[Proof of Theorem \ref{Thm:main} (\ref{Eq:LB})]
  Let $T = T_1 + T_2$, where $T_1 := \floor{n \log n - \b n}$ and $T_2
  := \floor{\frac{1}{2} n \log n - \b n}$. Fix any $\r \in G^\ast$. By
  Lemma \ref{Lem:Burn-lower} followed by Lemma \ref{Lem:DE-lower}, we
  have for large enough $\b$ that
  \[ \Pr_{\s_\star} \left(\s_T \in \Sast{\frac{\b}{\sqrt{n}}} \right) \le \Pr_{\s_\star}\left( \|x_\r(T)\|_{\HS} \le \sqrt{\frac{\Qc}{d_\r}}\frac{\b}{\sqrt{n}} \right) \le \frac{8\Qc^2}{\b^2}. \]
  On the other hand, Lemma \ref{Lem:stationary} tells us that
  \[ \pi\left(\Sc_\ast\left(\frac{\b}{\sqrt{n}}\right)\right) \ge 1-
  \frac{c_G}{\b^2}. \]
  Consequently,
  \[ d_{\s_\star}(T) \ge 1 - \frac{c_G}{\b^2} - \frac{8\Qc^2}{\b^2}, \]
  which tends to $1$ as $\b \rightarrow \infty$, establishing
  \eqref{Eq:LB}.
\end{proof}

\section*{Acknowledgements}

This work was initiated while R.T. and A.Z. were visiting Microsoft
Research in Redmond. They thank Microsoft Research for the
hospitality. R.T. was also visiting the University of Washington in
Seattle and thanks Professor Christopher Hoffman for making his visit
possible. R.T. is supported by JSPS Grant-in-Aid for Young
Scientists (B) 17K14178. A.Z. is supported by a Stanford Graduate Fellowship.

\bibliographystyle{alpha}
\bibliography{mix}

\end{document}